\RequirePackage{fix-cm}
\documentclass[smallcondensed]{svjour3}     
\smartqed  
\usepackage{graphicx}
\usepackage{amssymb}

\begin{document}

\title{An inexact Picard iteration method for absolute value equation\thanks{This work is supported by the Natural
Science Foundation of Gansu Province (Grant No. 145RJZA102, 145RJZA037, 145RJZA099) and the Youth Research Ability Project of
Northwest Normal University (Grant No. NWNU-LKQN-13-15).}
}


\author{Shu-Xin Miao\and Xiang-Tuan Xiong\and Jin Wen
}


\institute{S.-X. Miao\and X.T. Xiong\and J. Wen\at
              College of Mathematics and Statistics,\\Northwest Normal University, \\
              Lanzhou, 730070, P.R. China\\
              Tel.: +86-931-7971124\\
              \email{shuxinmiao@gmail.com}           
        }

\date{Received: date / Accepted: date}

\maketitle

\begin{abstract}
Recently, a class of inexact Picard iteration method for solving the absolute
value equation: $Ax-|x~|=b$ have been proposed in [Optim Lett 8:2191-2202,2014].  To further improve the performance of Picard iteration method, a new inexact Picard iteration
method is proposed to solve the absolute value equation. The sufficient
conditions for the convergence of the proposed method for the absolute
value equation is given. Some numerical experiments are given to demonstrate the effectiveness of the new method.
\end{abstract}

\section{Introduction}
\label{intro}
In this study, we consider the iteration method for solving the absolute value equation (AVE) of the form
\begin{equation}\label{1}
Ax-|x~|=b,
\end{equation}
where $A\in\mathbb{C}^{n\times n}$, $x,~b\in\mathbb{C}^n$, and $|x~|$ denotes the component-wise absolute value of the vector $x$, i.e., $|x~|=(|x_1|,\cdots,|x_n|)^T$. The AVE (\ref{1}) is a special case of the generalized absolute value equation of the type
\begin{equation}\label{2}
Ax+B|x|=b,
\end{equation}
where $B\in\mathbb{C}^{n\times n}$. The generalized absolute value equation (\ref{2}) was introduced in \cite{Rohn-04-LMA} and investigated in a more general context \cite{OM-12-OL,OMM-06-LAA,OP-09-COA}. The absolute value equation (\ref{1}) or (\ref{2}) arises in a variety of optimization problems, e.g. linear complementarity problem, linear programming or convex quadratic programming problems; see for example \cite{Hu-10-OL,OM-12-OL,OMM-06-LAA,OP-09-COA,Rohn-04-LMA}.

For AVE (\ref{1}), one can deduce that all singular values of $A$ exceeding 1 implies existence of a unique solution for every right-hand side $b$ \cite{OMM-06-LAA}. When AVE (\ref{1}) has the unique solution, how to find the solution of (\ref{1}) is a main research topic. In recent years, a large variety of methods for solving AVE (\ref{1}) can be found in the literature \cite{Qu-11-COA,OM-09-OL-a,OM-09-OL-b,No-12-OL,Rohn-14-OL,DS-14-OL,Zhang-09-JOTA}. Among these methods, Picard-type methods capture one's attention. Rohn et al. in \cite{Rohn-14-OL} proposed a class of method to solve AVE
(\ref{1}), in practice their method is reduced to the well known Picard iteration method
\begin{equation}\label{3}
x^{(k+1)}=A^{-1}\left(|x^{(k)}|+b\right),~~k=0,1,2,\cdots,
\end{equation}
where $x^{(0)}=A^{-1}b$ is the initial guess. From (\ref{3}), we can see that there is a linear system with the constant coefficient matrix $A$ should be solved in each iteration of the Picard method. To improve the performance of the Picard method, the linear system with matrix $A$ should be solved by inner iteration, this leads to inexact Picard iteration method. As an example, Salkuyeh suggested that using Hermitian and skew-Hermitian splitting iteration (HSS) method \cite{Bai-HSS} to approximation the solution of the linear system with $A$ at each Picard iteration, and proposed the Picard-HSS method for solving AVE(\ref{1}) \cite{DS-14-OL}. In fact, the Picard-HSS method has been proposed originally by Bai and Yang for weakly nonlinear systems in \cite{Bai-P-HSS}. The sufficient conditions to guarantee the convergence of the Picard-HSS method and some numerical experiments are given to show the effectiveness of the method for solving AVE(\ref{1}) in \cite{DS-14-OL}.

Bear in mind that there are two linear subsystem need to be solved at each step of the inner HSS iteration of the Picard-HSS method \cite{Bai-P-HSS,DS-14-OL}, one is the linear subsystem with shift Hermitian coefficient matrix and the other is the ones with shift skew-Hermitian coefficient matrix. The solution of linear subsystem with shift Hermitian coefficient matrix can be easily obtained by CG method, however, the solution of linear subsystem with shift skew-Hermitian coefficient matrix is not easy to obtain, in some cases, its solution is as difficult as that of the original linear system. To avoid solving a linear subsystem with shift coefficient skew-Hermitian in the inner iteration of the inexact Picard method, we use the single-step HSS method \cite{Wu} to approximate the solution of the linear system with coefficient matrix $A$ and present a new inexact Picard method, abbreviated as Picard-SHSS iteration method, in this paper.

The rest of this paper is organized as follows. In Section~\ref{S2}, after review some notes and the single-step HSS iteration method, the Picard-SHSS iteration method for solving AVE (\ref{1}) is described. And then the convergence properties of the Picard-SHSS iteration method is studied. Numerical experiments
are presented in Section~\ref{S3}, to show the feasibility and effectiveness of the Picard-SHSS method.

\section{The Picard-SHSS method}\label{S2}

For convenience, some notations, definitions and results that will be used in the
following parts are given below. For a matrix $A\in\mathbb{C}^{n\times n}$, $A^*$ represents the conjugate transpose of $A$, and $\rho(A)$ denotes the spectral
radius of $A$. $A$ is said to be non-Hermitian positive definite if its Hermitian part $H=\frac{1}{2}(A+A^*)$ is positive definite, i.e., $x^*Hx>0$ for any $x\in\mathbb{C}^n\backslash\{0\}$.

Let $A\in\mathbb{C}^{n\times n}$ be a non-Hermitian positive definite matrix, and $A=H+S$ be its Hermitian and skew-Hermitian splitting (HSS) with $$H=\frac{1}{2}(A+A^*)~~\mbox{and}~~S=\frac{1}{2}(A-A^*).$$ Based on the HSS of $A$, Bai et al. \cite{Bai-HSS} presented the HSS iteration method to solve non-Hermitian positive definite system of linear equations $Ax=q$. There are two linear subsystem need to be solved at each step of the HSS iteration method, one is the linear subsystem with shift Hermitian coefficient matrix $\alpha I+H$ and the other is the ones with shift skew-Hermitian coefficient matrix $\alpha I+S$ for any positive constant $\alpha$ and identity matrix $I$; see \cite{Bai-HSS} for more details. The challenges of the HSS iteration method lies in solving the linear subsystem with $\alpha I+S$, which is as difficult as that of the original linear system in some cases. To avoid solving a linear subsystem with $\alpha I+S$ in the HSS iteration method, the single-step HSS method is proposed recently \cite{Wu}. The iteration scheme of the single-step HSS method used for solving system of linear equations $Ax=q$ can be written equivalently as
\begin{equation}\label{SHSS}
(\alpha I+H)x^{(k+1)}=(\alpha I-S)x^{(k)}+q,
\end{equation}
here $\alpha$ is a positive iteration parameter. It has been proved that, under a loose restriction on
the iteration parameter $\alpha$, the single-step HSS method is convergent to the unique solution of the linear
system $Ax=q$ for any initial guess $x^{(0)}\in \mathbb{C}^n$; see \cite{Wu}.

Recalling that the Picard iterative method for solving AVE (\ref{1}) is a fixed-point iterative method of the form
\begin{equation}\label{7}
Ax^{(k+1)}=|x^{(k)}|+b,~~k=0,1,2,\cdots.
\end{equation}
We assume that the matrix $A$ is non-Hermitian positive definite. In this case, the next
iterate of $x^{(k+1)}$ can be approximately computed by the single-step HSS iteration by making use
of the splitting $A=M(\alpha)-N(\alpha)$ as following (see \cite{Bai-P-HSS})
\begin{equation}\label{8}
\begin{array}{c}
M(\alpha)x^{(k,l+1)}=N(\alpha)x^{(k,l)}+|x^{(k)}|+b,\\
\begin{array}{l}l=0,1,\cdots,l_k-1,\\
k=0,1,2,\cdots,
\end{array}
\end{array}
\end{equation}
where $M(\alpha)=\alpha I+H$ and $N(\alpha)=\alpha I-S$, $\alpha$ is a positive constant, $\{l_k\}_{k=0}^\infty$ a prescribed sequence of positive integers, and $x^{(k,0)}=x^{(k)}$ is the starting point of the inner single-step HSS iteration at $k$-th outer Picard iteration. This leads to the inexact
Picard iteration method, called Picard-SHSS iteration method, for solving the system (\ref{1})
which can be summarized as following

\noindent{\bf The Picard-SHSS iteration method}: {\it Let $A\in\mathbb{C}^{n\times n}$ be no-Hermitian positive definite, $H=\frac{1}{2}(A+A^*)$ and $S=\frac{1}{2}(A-A^*)$ be the Hermitian and skew-Hermitian parts of $A$ respectively. Given an initial guess $x^{(0)}\in\mathbb{C}^n$ and a sequence $\{l_k\}_{k=0}^\infty$ of positive integers, compute $x^{(k+1)}$ for $k=0,1,2,\cdots$ using the following iteration scheme until $\{x^{(k)}\}$ satisfies the stopping criterion:
\begin{itemize}
\item[(a).] Set $x^{(k,0)}=x^{(k)}$;
\item[(b).] For $l=0,1,\cdots,l_k-1$, solve the following linear system to obtain $x^{(k,l+1)}$:
$$(\alpha I+H)x^{(k,l+1)}=(\alpha I-S)x^{(k,l)}+|x^{(k)}|+q,$$
where $\alpha$ is a positive constant and $I$ is a the identity matrix;
\item[(c).] Set $x^{(k+1)}=x^{(k,l_k)}$.
\end{itemize}}

Compared with the Picard-HSS iteration method studied in \cite{DS-14-OL}, a linear subsystem with $\alpha I+S$ is avoided in the inner iteration of the Picard-SHSS iteration method. The involved linear subsystem with $\alpha I+H$ of the Picard-SHSS iteration method can be efficiently solved exactly by a sparse Cholesky factorization, or inexactly by a preconditioned Conjugate Gradient method \cite{Saad}.

The next theorem provides sufficient conditions for the convergence of the Picard-SHSS method to solve the AVE (\ref{1}).

\begin{theorem}\label{T1}
Let $A\in\mathbb{C}^{n\times n}$ be a non-Hermitian positive definite matrix and $H=\frac{1}{2}(A+A^*)$ and $S=\frac{1}{2}(A-A^*)$ be its Hermitian and skew-Hermitian parts, respectively. Let $\alpha$ be a constant number such that $\alpha>\max\left\{0,\frac{\sigma_{\max}^2-\lambda_{\min}^2}{2\lambda_{\min}}\right\}$, where $\lambda_{\min}$ is the smallest eigenvalue of $H$ and $\sigma_{\max}$ is the largest singular-value of $S$. Let also $\eta=\|A^{-1}\|_2<1$. Then the AVE (\ref{1}) has a unique solution $x^*$, and for any initial guess $x^{(0)}\in\mathbb{C}^n$ and any sequence of positive integers $l_k$, $k=0,1,\cdots$, the iteration sequence $\{x^{(k)}\}_{k=0}^\infty$ produced by the Picard-SHSS iteration method converges to $x^*$ provided that $l=\lim\inf_{k\rightarrow\infty}l_k\geq N$, where $N$ is a natural number satisfying $$\|T(\alpha)^s\|_2<\frac{1-\eta}{1+\eta}~~\forall s\geq N.$$
\end{theorem}

\begin{proof}
Let $T(\alpha)=M(\alpha)^{-1}N(\alpha)$, based on the iteration scheme (\ref{8}), we can express the $(k+1)$th iterate $x^{(k+1)}$ of the Picard-SHSS iteration method as
\begin{equation}\label{10}
x^{(k+1)}=T(\alpha)^{l_k}x^{(k)}+\sum_{j=0}^{l_k-1}T(\alpha)^{j}M(\alpha)^{-1}(|x^{(k)}|+b),~k=0,1,2,\cdots.
\end{equation}
Note that $\eta<1$, then AVE (\ref{1}) has a unique solution $x^*\in\mathbb{C}^n$ \cite{OMM-06-LAA} such that
\begin{equation}\label{11}
x^*=T(\alpha)^{l_k}x^*+\sum_{j=0}^{l_k-1}T(\alpha)^{j}M(\alpha)^{-1}(|x^*|+b),~k=0,1,2,\cdots.
\end{equation}
By subtracting (\ref{11}) from (\ref{10}) we have
\begin{equation}\label{12}
x^{(k+1)}-x^*=T(\alpha)^{l_k}(x^{(k)}-x^*)+\sum_{j=0}^{l_k-1}T(\alpha)^{j}M(\alpha)^{-1}(|x^{(k)}|-x^*).
\end{equation}

It follows from \cite[Theorem 2.1]{Wu} that $\rho(T(\alpha))<1$ when $\alpha$ satisfying $\alpha>\max\left\{0,\frac{\sigma_{\max}^2-\lambda_{\min}^2}{2\lambda_{\min}}\right\}$. In this case, some calculations yield $$\sum_{j=0}^{l_k-1}T(\alpha)^{j}M(\alpha)^{-1}=\left(I-T(\alpha)^{l_k}\right)A^{-1}.$$
Now (\ref{12}) becomes
\begin{eqnarray*}
x^{(k+1)}-x^*&=&T(\alpha)^{l_k}(x^{(k)}-x^*)+\left(I-T(\alpha)^{l_k}\right)A^{-1}(|x^{(k)}|-x^*)\\
&=&T(\alpha)^{l_k}\left[(x^{(k)}-x^*)-A^{-1}(|x^{(k)}|-x^*)\right]+A^{-1}(|x^{(k)}|-x^*).
\end{eqnarray*}
Note that $\||x|-|y|\|_2\leq\|x-y\|_2$ for any $x,y\in\mathbb{C}^n$, it then follows that $$\left\|x^{(k+1)}-x^*\right\|_2\leq\left(\left\|T(\alpha)^{l_k}\right\|_2(1+\eta)+\eta\right)\left\|x^{(k)}-x^*\right\|_2.$$

The condition of $\rho(T(\alpha))<1$ when $\alpha$ satisfying $\alpha>\max\left\{0,\frac{\sigma_{\max}^2-\lambda_{\min}^2}{2\lambda_{\min}}\right\}$ ensure that $T(\alpha)$ tend to 0 as $s$ tend  to infinity. Therefore, there is a natural number $N$ such that
$$\left\|T(\alpha)^{s}\right\|_2<\varepsilon :=\frac{1-\eta}{1+\eta}~~~\forall s\geq N.$$

Now, if we let $l=\lim\inf_{k\rightarrow\infty}l_k\geq N$, then $\left\|x^{(k+1)}-x^*\right\|_2<\left\|x^{(k)}-x^*\right\|_2$, hence the iteration sequence $\{x^{(k)}\}_{k=0}^\infty$ produced by the Picard-SHSS iteration method converges to $x^*$. $\hfill\Box$
\end{proof}

In actual computation, the residual-updating form of the Picard-SHSS iteration method is more convenient, which can
be written as following.

\noindent{\bf The Picard-SHSS iteration method} (residual-updating variant): {\it Let $A\in\mathbb{C}^{n\times n}$ be no-Hermitian positive definite, $H=\frac{1}{2}(A+A^*)$ and $S=\frac{1}{2}(A-A^*)$ be the Hermitian and skew-Hermitian parts of $A$ respectively. Given an initial guess $x^{(0)}\in\mathbb{C}^n$ and a sequence $\{l_k\}_{k=0}^\infty$ of positive integers, compute $x^{(k+1)}$ for $k=0,1,2,\cdots$ using the following iteration scheme until $\{x^{(k)}\}$ satisfies the stopping criterion:
\begin{itemize}
\item[(a).] Set $s^{(k,0)}=0$ and $b^{(k)}=|x^{(k)}|+b-Ax^{(k)}$;
\item[(b).] For $l=0,1,\cdots,l_k-1$, solve the following linear system to obtain $s^{(k,l+1)}$:
$$(\alpha I+H)s^{(k,l+1)}=(\alpha I-S)s^{(k,l)}+b^{(k)},$$
where $\alpha$ is a positive constant and $I$ is the identity matrix;
\item[(c).] Set $x^{(k+1)}=x^{(k)}+s^{(k,l_k)}$.
\end{itemize}}

\section{Numerical experiments}\label{S3}

In this section we give some numerical experiments to show the effectiveness of the
Picard-SHSS iteration method to solve AVE (\ref{1}), to do this, the numerical properties of the Picard-HSS and Picard-SHSS methods are examined and compared experimentally by a suit of test problems. We use the residual-updating versions of the Picard-HSS iteration method \cite{DS-14-OL} and Picard-SHSS iteration method.

All the numerical experiments presented in this section have been computed in double precision using some MATLAB R2012b on Intel(R) Core(TM) i5-2400 CPU 3.10 GHz and 4.00 GB of RAM. All runs are started from the initial zero
vector and terminated if the current relative residual satisfies
$${\rm RES}:=\frac{\|Ax^{(k)}-|x^{(k)}|-b\|_2}{\|b\|_2}\leq 10^{-7},$$
where $x^{(k)}$ is the computed solution by each of the methods at iteration $k$, and a maximum number of the iterations
500 is used. In addition, the stopping criterion for the inner iterations of the Picard-HSS and Picard-SHSS methods are set to be
$$\frac{\|b^{(k)-As^{(k,l)}}\|_2}{\|b^{(k)}\|_2}\leq 0.01,$$
and a maximum number of the iterations 10 ($l_k=10,~k=0,1,2,\cdots$) for inner iterations are used. The right-hand side vector of AVE (\ref{1}) is taken such a way that the vector $x=(x_1,~x_2,~\cdots,~x_n)^T$ with $$x_i=(-1)^ii,~~i=1,2,\cdots,n$$
be the exact solution.

The optimal parameters employed in the Picard-HSS and Picard-SHSS iteration methods are chosen to be the experimentally found optimal ones, which result in the least
number of iteration steps of iteration methods.

The coefficient matrix $A$ of AVE (\ref{1}) is given by
\begin{equation}\label{9}
A=T_x\otimes I_m+I_m\otimes T_y+pI_n,
\end{equation}
where $I_m$ and $I_n$ are the identity matrices of order $m$ and $n$ with $n=m^2$, $\otimes$ means the Kronecker product, $T_x$ and $T_y$ are tridiagonal matrices
$$T_x=\mbox{tridiag}(t_2,~t_1,~t_3)_{m\times m}~~\mbox{and}~~T_y=\mbox{tridiag}(t_2,~0,~t_3)_{m\times m}$$
with $t_1=4$,~$t_2=-1-Re$, $t_3=-1+Re$. Here $Re=(qh)/2$ and $h=1/(m+1)$ are the mesh Reynolds number and the equidistant step size, respectively, and $q$ is a positive constant. In fact, the matrix $A$ arising from the finite difference approximation the two-dimensional convection-diffusion equation
$$
\left\{
\begin{array}{ll}
-(u_{xx}+u_{yy})+q(u_x+u_y)+pu=f(x,y),&~(x,y)\in\Omega,\\
u(x,y)=0,&~(x,y)\in\partial\Omega,
\end{array}\right.$$
where $\Omega=(0,~1)\times(0,~1)$, $\partial\Omega$ is its boundary, $q$ is a positive constant used to measure the magnitude of the diffusive term and $p$ is a real number. If we use the five-point finite difference scheme to the diffusive terms and the central difference scheme to the convective terms, then we obtained the matrix $A$. It is easy to find that for every nonnegative number $q$ the matrix $A$ is in general non-symmetric positive definite \cite{DS-14-OL}.

In our numerical experiments, the matrix $A$ in AVE (\ref{1}) is defined by (\ref{9}) with different values of $q$ ($q=0,~1,~10,~\mbox{and}~100$) and different values of $p$ ($p=0~\mbox{and}~-1$). In Table~\ref{t3} and Table~\ref{t4}, we present the numerical results with respect to the Picard-HSS and Picard-SHSS iteration methods, the experimentally optimal parameters used in the Picard-HSS and Picard-SHSS iteration methods are those given in Table~\ref{t1} and Table~\ref{t2}. We give the elapsed CPU time in seconds for the convergence (denoted by CPU), the number of iterations for the convergence (denoted by IT) and the relative residuals (denoted by RES).

From the Table~\ref{t3} and Table~\ref{t4}, we see that both the Picard-HSS and Picard-SHSS iteration methods can successfully produced
approximate solution to the AVE (\ref{1}) for all of the problem-scales $n=m^2$ and the convective measurements $q$. For the convergent cases, the CPU time also increases rapidly with the increasing of the problem-scale for all tested iteration methods. Moreover, numerical results in the two tables show that the Picard-SHSS iteration method perform better than the Picard-HSS iteration method in most cases as the former one cost the least CPU time to achieve stopping criterion except the case of $q=100$ and $n=100,400$. In addition, for $p=-1$, the Picard-SHSS iteration method costs the least number of iteration steps and CPU time to achieve stopping criterion. In summary, the Picard-SHSS iteration method is useful and effective for solving the NP-hard AVE (\ref{1}).

\begin{acknowledgements}
We would like to thank Professor Davod Khojasteh Salkuyeh from University of Guilan for providing us the MATLAB code of the Picard-HSS method.
\end{acknowledgements}


\newcommand{\nosort}[1]{}


\newpage
\begin{table}[ht]
\caption{The optimal parameters for Picard-HSS and Picard-SHSS methods ($p=0$)} \label{t1}
\begin{center}
\begin{tabular}{cccccc}
\hline\noalign{\smallskip}
\multicolumn{2}{c}{Optimal parameters}&$m=10$& $m=20$ &$m=40$&$m=80$\\
\hline\noalign{\smallskip}
$q=0$&Picard-HSS&$11.69$&$12.6$&$13.4$&$13$\\
&Picard-SHSS&$5.745$&$6.5$&$6.4$&$6.6$\\
$q=1$&Picard-HSS&$12.01$&$13.6$&$14$&$13$\\
&Picard-SHSS&$5.926$&$6.6$&$6.75$&$6.6$\\
$q=10$&Picard-HSS&$6.76$&$10.99$&$13.43$&$15.8$\\
&Picard-SHSS&$3.594$&$5.52$&$6.63$&$8.0$\\
$q=100$&Picard-HSS&$23.3$&$23.1$&$8.7$&$9.2$\\
&Picard-SHSS&$81.5$&$26.4$&$4.99$&$4.57$\\
\noalign{\smallskip}\hline
\end{tabular}
\end{center}
\end{table}

\begin{table}[ht]
\begin{center}
\caption{Numerical results for different values of $m$ and $q$ ($p=0$)} \label{t3}
\begin{tabular}{ccccccc}
\hline\noalign{\smallskip}
\multicolumn{3}{c}{Methods}&$m=10$& $m=20$ &$m=40$&$m=80$\\
\hline\noalign{\smallskip}
$q=0$&Picard-HSS&IT&$36$&$32$&$30$&$28$\\
&&CPU&$0.0280$&$0.0371$&$0.1266$&$0.9486$\\
&&RES&$9.8815e-8$&$9.3643e-8$&$9.9900e-8$&$9.9043e-8$\\
&Picard-SHSS&IT&$37$&$32$&$30$&$29$\\
&&CPU&$0.0137$&$0.0245$&$0.1075$&$0.9060$\\
&&RES&$9.4432e-8$&$9.6857e-8$&$9.7569e-8$&$9.9256e-8$\\
\hline\noalign{\smallskip}
$q=1$&Picard-HSS&IT&$35$&$32$&$31$&$28$\\
&&CPU&$0.0220$&$0.0457$&$0.2377$&$1.9978$\\
&&RES&$9.1372e-8$&$9.8614e-8$&$9.3177e-8$&$9.7012e-8$\\
&Picard-SHSS&IT&$36$&$32$&$31$&$29$\\
&&CPU&$0.0136$&$0.0273$&$0.1292$&$1.1940$\\
&&RES&$9.8685e-8$&$9.4248e-8$&$9.4655e-8$&$9.6233e-8$\\
\hline\noalign{\smallskip}
$q=10$&Picard-HSS&IT&$29$&$66$&$33$&$36$\\
&&CPU&$0.0163$&$0.0939$&$0.2636$&$2.5265$\\
&&RES&$9.5635e-8$&$9.9755e-8$&$9.8395e-8$&$9.9024e-8$\\
&Picard-SHSS&IT&$29$&$63$&$33$&$37$\\
&&CPU&$0.0105$&$0.0528$&$0.1372$&$1.4083$\\
&&RES&$9.6650e-8$&$9.8439e-8$&$9.7901e-8$&$9.8170e-8$\\
\hline\noalign{\smallskip}
$q=100$&Picard-HSS&IT&$11$&$17$&$35$&$146$\\
&&CPU&$0.0117$&$0.0342$&$0.2826$&$9.6469$\\
&&RES&$9.8396e-8$&$9.9503e-8$&$9.3177e-8$&$9.7165e-8$\\
&Picard-SHSS&IT&$44$&$25$&$35$&$140$\\
&&CPU&$0.0161$&$0.0246$&$0.1471$&$4.9182$\\
&&RES&$9.8783e-8$&$9.8253e-8$&$9.8972e-8$&$9.7299e-8$\\
\noalign{\smallskip}\hline
\end{tabular}
\end{center}
\end{table}

\newpage
\begin{table}[ht]
\caption{The optimal parameters for Picard-HSS and Picard-SHSS methods ($p=-1$)} \label{t2}
\begin{center}
\begin{tabular}{cccccc}
\hline\noalign{\smallskip}
\multicolumn{2}{c}{Optimal parameters}&$m=10$& $m=20$ &$m=40$&$m=80$\\
\hline\noalign{\smallskip}
$q=0$&Picard-HSS&$13.8$&$11.2$&$10.36$&$10.1$\\
&Picard-SHSS&$6.96$&$5.7$&$5.21$&$5.1$\\
$q=1$&Picard-HSS&$14$&$11.29$&$10.4$&$11$\\
&Picard-SHSS&$7.05$&$5.72$&$5.2$&$5.1$\\
$q=10$&Picard-HSS&$20.55$&$13.6$&$11.1$&$11$\\
&Picard-PHSS&$11.2$&$7.0$&$5.65$&$5.2$\\
$q=100$&Picard-HSS&$27.2$&$22$&$11.1$&$21$\\
&Picard-SHSS&$108$&$33.1$&$11.57$&$10.65$\\
\noalign{\smallskip}\hline
\end{tabular}
\end{center}
\end{table}

\begin{table}[ht]
\begin{center}
\caption{Numerical results for different values of $m$ and $q$ ($p=-1$)} \label{t4}
\begin{tabular}{ccccccc}
\hline\noalign{\smallskip}
\multicolumn{3}{c}{Methods}&$m=10$& $m=20$ &$m=40$&$m=80$\\
\hline\noalign{\smallskip}
$q=0$&Picard-HSS&IT&$24$&$62$&$206$&$769$\\
&&CPU&$0.0186$&$0.0944$&$0.8224$&$25.7551$\\
&&RES&$9.6164e-8$&$9.5608e-8$&$9.8436e-8$&$9.8404e-8$\\
&Picard-SHSS&IT&$21$&$52$&$166$&$614$\\
&&CPU&$0.0099$&$0.0444$&$0.5608$&$18.5361$\\
&&RES&$9.6537e-8$&$9.7425e-8$&$9.9905e-8$&$9.8795e-8$\\
\hline\noalign{\smallskip}
$q=1$&Picard-HSS&IT&$24$&$61$&$203$&$908$\\
&&CPU&$0.0193$&$0.0877$&$1.4320$&$57.8838$\\
&&RES&$9.7842e-8$&$9.9013e-8$&$9.3292e-8$&$9.9324e-8$\\
&Picard-SHSS&IT&$21$&$51$&$166$&$598$\\
&&CPU&$0.0102$&$0.0421$&$0.6038$&$18.9311$\\
&&RES&$9.5087e-8$&$9.5556e-8$&$9.5308e-8$&$9.8338e-8$\\
\hline\noalign{\smallskip}
$q=10$&Picard-HSS&IT&$14$&$27$&$74$&$297$\\
&&CPU&$0.0175$&$0.0519$&$0.5532$&$18.6613$\\
&&RES&$9.8754e-8$&$9.1215e-8$&$9.8943e-8$&$9.6129e-8$\\
&Picard-SHSS&IT&$14$&$23$&$61$&$204$\\
&&CPU&$0.0084$&$0.0283$&$0.2427$&$6.8279$\\
&&RES&$9.0191e-8$&$9.9463e-8$&$9.3622e-8$&$9.7081e-8$\\
\hline\noalign{\smallskip}
$q=100$&Picard-HSS&IT&$14$&$20$&$64$&$65$\\
&&CPU&$0.0151$&$0.0372$&$0.4811$&$4.2747$\\
&&RES&$9.8364e-8$&$9.8128e-8$&$9.4024e-8$&$9.8372e-8$\\
&Picard-SHSS&IT&$83$&$44$&$71$&$63$\\
&&CPU&$0.0508$&$0.0384$&$0.2824$&$2.2035$\\
&&RES&$9.9168e-8$&$9.9270e-8$&$9.8733e-8$&$9.5253e-8$\\
\noalign{\smallskip}\hline
\end{tabular}
\end{center}
\end{table}
\end{document}